\documentclass{amsart}

\usepackage{amsmath}
\usepackage{amsfonts,txfonts}
\usepackage{amsthm}
\usepackage{graphicx} 
\usepackage{amssymb,mathrsfs}
\usepackage{url}
\usepackage{color}
\usepackage{xcolor} 

\newtheorem{theorem}{Theorem}[section]
\newtheorem{lemma}[theorem]{Lemma}
\newtheorem{corollary}[theorem]{Corollary}
\newtheorem{proposition}[theorem]{Proposition}

\theoremstyle{definition}
\newtheorem{definition}[theorem]{Definition}
\newtheorem{algorithm}[theorem]{Algorithm}

\theoremstyle{remark}
\newtheorem{remark}[theorem]{Remark}

\numberwithin{equation}{section}

\newcommand{\bN}{\mathbb{bN}}
\newcommand{\bR}{\mathbb{R}}
\newcommand{\bQ}{\mathbb{Q}}
\newcommand{\bC}{\mathbb{C}}

\newcommand{\sC}{\mathscr{C}}

\newcommand{\NN}{\mathbb{N}}

\newcommand{\CC}{\mathbb{C}}

\newcommand{\Ct}{\mathbb{C}^2}

\newcommand{\ep}{\varepsilon}

\newcommand{\cN}{\mathcal{N}}

\newcommand{\cU}{\mathcal{U}}

\newcommand{\cA}{\mathcal{A}}
\newcommand{\cW}{\mathcal{W}}

\newcommand{\cB}{\mathcal{B}}

\newcommand{\cBsurvk}{\mathcal{B}_k^{\mathrm{surv}}}
\newcommand{\Wslocalpha}{W^s_{\mathrm{loc}}(\alpha)}

\renewcommand{\Re}{\text{Re}}
\renewcommand{\Im}{\text{Im}}

\newcommand{\Henon}{H\'{e}non }

\newcommand{\HNNP}{\mathcal{H}_{\mathcal{C}}}

\begin{document}

\title[Attracting and neutral dynamics and computability of Julia sets for H\'enon maps]
{Computability of Julia sets for complex H\'enon maps: The role of attracting and neutral cycles}
 
\author[S. Boyd]{Suzanne Boyd}
\address{Department of Mathematical Sciences\\
University of Wisconsin Milwaukee\\
PO Box 413\\
Milwaukee, WI 53201, 
USA}
\email{sboyd@uwm.edu, ORCID: 0000-0002-9480-4848}

\author[C. Wolf]{Christian Wolf}
\address{Department of Mathematics and Statistics\\
Mississippi State University\\
Starkville, MS 39759, USA}
\email{cwolf@math.msstate.edu, ORCID: 0000-0002-7976-3574.}
\thanks{C.W.\ was partially supported by a grant from the Simons Foundation (SFI-MPS-TSM-00013897). A significant portion of the paper was developed during a visit of S.B.\  to the Department of Mathematics and Statistics of Mississippi State University. We thank the department for their hospitality and their support of the visit.}

\date{\today}

\begin{abstract}
We study the computability of Julia sets for polynomial diffeomorphisms of $\mathbb{C}^2$ with dynamical degree $d>1$, whose prototypical examples are complex \Henon maps. In previous work, we established computability under the assumption of hyperbolicity (Axiom~A). Here, we extend this result to maps whose Fatou components are attracting basins, allowing for the possibility of no attracting basins or infinitely many basins. This yields computability of the Julia set for several classes of non-hyperbolic maps, including certain substantially dissipative maps in the Lyubich-Peters class, and certain quasi-hyperbolic maps.
Our proof is based on an algorithm that separates the dynamics into escaping, attracting, and saddle regimes. A key ingredient is the use of stable manifolds of saddle periodic points to approximate the forward Julia set via backward iteration. We first prove the result for generalized \Henon mappings and then extend it to arbitrary polynomial diffeomorphisms by expressing them as finite compositions of generalized \Henon maps. 
Finally, we present examples of non-computability in the presence of neutral dynamics, including \Henon maps with computable coefficients exhibiting semi-Siegel behavior. These examples show that the computability/non-computability dichotomy associated with neutral dynamics in one-dimensional complex dynamics in part persists in higher dimensions.

\end{abstract}

\maketitle

\section{Introduction}
\label{sec:Intro}

The computability of dynamically defined sets, and in particular Julia sets, has been extensively studied in one-dimensional complex dynamics over the past couple of decades, see, e.g., (\cite{BBRY-2011,Braverman2005,Bravermanparabolic, BY2009,BSW2020,CRY2018}).  It is known that Julia sets of polynomials in the complex plane are computable provided they have no Siegel disks; as a complementary result, Braverman and Yampolsky show that there exist Siegel disk polynomials with computable coefficients whose Julia sets are not computable (\cite{BravermanYampolsky2006, BY2009}). Moreover, although hyperbolic and parabolic polynomials have polynomial-time computable Julia sets \cite{Braverman2005, Bravermanparabolic}, in the parabolic case the algorithm requires the input of some basic combinatorial information about the parabolic point.  
Thus, in one complex dimension, computability is closely tied to the underlying dynamical structure of neutral periodic orbits.

In a sequence of recent works \cite{BoydWolf-Skew1,BoydWolf-1Dim,BoydWolf-Henon}, we initiated the study of computability questions in higher-dimensional complex dynamics. In particular,  under hyperbolicity (or Axiom~A) assumptions, we established computability of Julia sets for polynomial skew products and for polynomial diffeomorphisms of $\mathbb{C}^2$ (including complex H\'enon mappings). Our results provided the first computability results for Julia sets in several complex variables.

In the present paper, we continue this program considering for the first time the computability of the Julia set of higher dimensional diffeomorphisms without explicitly assuming hyperbolicity. Let $f$ be a polynomial diffeomorphism of $
\mathbb{C}^2$ with dynamical degree $d>1$, where 
\[
d = d(f) = \lim_{n \to \infty} \textit{deg}(f^{n})^{1/n}. 
\]
These are the polynomial diffeomorphisms of $\mathbb{C}^2$ with interesting dynamics, see \cite{FM} for details. Recall that any such diffeomorphism is conjugate to a finite composition of generalized \Henon maps. We refer to Section 2 for more details about \Henon maps.

There are several invariant sets that capture the basic features of $f$, analogous to the case of polynomials in $\bC$. Namely, let $K^+$ resp.\ $K^-$ be the set of points whose forward resp.\ backward orbits are bounded. Define $J^{\pm} = \partial K^{\pm}$, $J=J^+\cap J^-$ and $K=K^+\cap K^-$. Both $J$ and $K$ are compact. The set $J$ is called the Julia set of $f$. 
 Note that since we study orbits of $f$ and $f^{-1}$ simultaneously, it is sufficient to consider the volume-preserving case ($|{\rm det}\, Df|=1)$ and volume-decreasing (dissipative) case ($|{\rm det}\, Df|<1$) because otherwise we could simply consider $f^{-1}$. 
A connected component $U$ of ${\rm int}\, K^+$ is called a Fatou component. A Fatou component is either periodic ($f^k(U)=U$ for some $k$) or wandering in which case $f^k(U)\cap f^l(U)=\emptyset$ for all $k,l\geq 0, k\not=l$. A periodic Fatou component is called an attracting basin if it contains an attracting periodic point $\alpha$ (of period $k$). In this case $\Vert f^n(z)-f^n(\alpha)\Vert \to 0$ for all $z\in U$.

We are now ready to present the main result of the paper.

\begin{theorem} \label{thm:main}
Let $f$ be a polynomial diffeomorphism of $\mathbb{C}^2$ with dynamical degree $d>1$ such that all Fatou components are attracting basins. Then $J_f$ is computable.
\end{theorem}
We note that Theorem \ref{thm:main} incudes the cases of no Fatou components as well as infinitely many attracting basins. 

 It is an immediate consequence of Bedford and Smillie's classical structure result \cite{BS1} that Theorem~\ref{thm:main} implies the computability of hyperbolic Julia sets. We recall that this result was previously proven in \cite{BoydWolf-Henon} using different methods.
Moreover, Theorem~\ref{thm:main} applies to several important classes of non-hyperbolic maps. For instance, it applies to substantially dissipative \Henon maps with dominated splitting and no parabolic cycles, by the classification theorem of Lyubich and Peters \cite{LyubichPeters2021}. It also applies to other settings, such as the quasi-hyperbolic \Henon maps of Bedford and Smillie (\cite{BGS, BS8}), provided there are no wandering domains. In addition, the theorem could potentially apply to certain polynomial diffeomorphisms of $\mathbb{C}^2$ with infinitely many attracting cycles \cite{GregBuzzard}.

On the other hand, there are several explicit obstructions for the applicability of Theorem~\ref{thm:main}. In particular, it does not apply to maps with periodic Fatou components associated with neutral periodic points. Here, by neutral we mean that at least one eigenvalue of $Df^k(\alpha)$ has modulus one, where $\alpha$ is a periodic point of period $k$. Such periodic Fatou components include parabolic domains and Siegel cylinders (see Section~4 for a detailed discussion). Other potential obstructions could be the existence of Herman cylinders as well as wandering Fatou domains, although it is currently not known if either of these types of Fatou domains exist for polynomial diffeomorphisms of $\mathbb{C}^2$. Since unlike for polynomials in $\mathbb{C}$, there is currently no  classification result for the Fatou components of polynomial diffeomorphisms of $\mathbb{C}^2$ available, in principle there could be other obstructions preventing us from applying Theorem \ref{thm:main}. This is the reason why it is currently not expected to derive a classification for computability of the Julia set for all polynomial diffeomorphisms of $\mathbb{C}^2$. 

 Finally, we note that the methods used in the proof of Theorem~\ref{thm:main} are compatible with other types of periodic Fatou components, provided that these components can be computationally detected (or are given as inputs into the algorithm) and are lower semi-computable. In particular, we anticipate that the theorem might be extendable to include maps with parabolic domains, analogous to the corresponding results in one-dimensional complex dynamics.

Since any computable function must necessarily be continuous, as a consequence of Theorem~\ref{thm:main}  we obtain the following.

\begin{corollary} \label{cor:main}
Let $f_0$ be a polynomial diffeomorphism of $\mathbb{C}^2$ with dynamical degree $d>1$ such that all Fatou components of $f_0$ are attracting basins.
Then
\[
f\mapsto J_f
\]
varies continuously at $f_0$  with respect to the Hausdorff metric.
\end{corollary}

To prove Theorem~\ref{thm:main}, we first establish the result for generalized complex \Henon mappings (see  Section 2 for the definition), then using the fact that every polynomial diffeomorphism of $\mathbb{C}^2$ with dynamical degree $d>1$ is conjugate to a finite composition of generalized \Henon mappings \cite{FM}.

The main argument is constructing an algorithm which separates the dynamics into three regimes: escaping points, attracting basins, and saddle dynamics. This is why we need, in the case that $K^+$ has interior, all Fatou components to be basins of attraction.   The escaping set is detected via forward iteration, while attracting basins are identified through convergence. In the absence of attracting cycles, the algorithm exploits the saddle structure by computing local stable manifolds and approximating $J^+$ via their backward iterates. These mechanisms together yield arbitrarily accurate approximations of $J^+$ and $J^-$, and hence of $J = J^+ \cap J^-$. 
Moreover, the construction yields computable approximations of $K$, as well as of $J^+$ and $K^+$ restricted to any bounded region in $\mathbb{C}^2$.

Finally, the existence of non-attracting periodic Fatou components,
in particular Siegel cylinders, is an essential
obstruction to the current approach,
and begins to establish the boundary between computability and non-computability in our setting. 
In Section~\ref{sec:non-computability}, we 
obtain non-computability results for polynomial skew products and H\'enon mappings exhibiting semi-Siegel dynamics. We also show that there exist semi-Siegel  \Henon maps whose parameters are computable. Thus, we parallel the one-dimensional theory and illustrate the role of neutral behavior in obstructing computability.

The paper is organized as follows. In Section~\ref{sec:prelim}, we review the necessary background on computability and \Henon dynamics. Section~\ref{sec:results} contains the proof of Theorem~\ref{thm:main2}, followed by Section~\ref{sec:generalizations} containing the extension to Theorem~\ref{thm:main}.
In Section~\ref{sec:non-computability}, we present examples illustrating the role of neutral dynamics in the failure of computability. Finally, we close with suggestions for further study in Section~\ref{sec:future}.

\section{Background and Preliminaries}
\label{sec:prelim}

In this section, we provide a brief overview of the essential definitions and properties of the material required in this paper: first on computability in general, and then on \Henon mappings. For a more detailed account of computability in complex dynamics, particularly as it relates to the computability of Julia sets of maps of two complex variables, see \cite{BoydWolf-Skew1,BoydWolf-Henon}.

\subsection{Computability}
\label{sec:compute:basic}
We briefly recall the basic notions of computability used in this paper; see 
\cite{Braverman2005,BY2009,BSW2020} for detailed treatments.

We say a function $\psi:\bN\to\bQ^\ell$ is an oracle of $x\in \bR^\ell$ if $
\|\psi(n) - x\| < 2^{-n}.
$ 
Moreover,  $x \in \mathbb{R}^\ell$ is \emph{computable} if there exists a Turing machine $T=T(n)$ which is an oracle of $x$. We note that since the set of all Turing machines is countable, most points in $\bR^\ell$ are not computable.
Identifying $\mathbb{C}^\ell$ with $\mathbb{R}^{2\ell}$ provides the corresponding notion of computability in $\mathbb{C}^\ell$.
\begin{definition}
Let $D\subset \bR^\ell$. We say
a function $f : D  \to \mathbb{R}^m$ is \emph{computable} if there is a Turing machine $T=T(\psi,n)$  that, given $x\in D$ and any oracle $\psi$ of $x$, outputs $T(\psi,n)\in\bQ^m$ such that $\|T(\psi,n)-f(x)\|<2^{-n}$.
\end{definition}
We note that in the above definition the oracle $\psi$ of $x$ is an input of the Turing machine $T$, whereas the Turing machine must have the ability to algorithmically decide when the approximation $\psi(n)$ of $x$ is good enough to guarantee the computation of $f(x)$  at precision $2^{-n}$. In this situation the Turing machine $T$ is often called an oracle Turing machine. It is a straight-forward consequence of the definition that computable functions must be continuous.

Next we introduce a computability notion for the space of compact subsets of $\mathbb{C}^m$ endowed with the Hausdorff metric, see, e.g., \cite{BY2009}.

\subsection*{Computable subsets of $\mathbb{C}^m$}

Let $\sC_m$ denote the space of all compact subsets  of $ \mathbb{C}^m$ together with the Hausdorff metric defined by
\[
d_H(A,B) = \max\left\{\sup_{a \in A} d(a,B), \sup_{b \in B} d(b,A)\right\}.
\]
We recall that $\sC_m$ is a complete separable  metric space.
\begin{definition}
We say $C\in \sC_m$ is \emph{computable} if there exists a Turing machine which, on input $n$, outputs a finite union of dyadic boxes $\Psi_n$ such that
\[
d_H(C, \Psi_n) \leq 2^{-n}.
\]
The set $\Psi_n$ is called a {\em $2^n$-approximation of $C$}.
\end{definition}
Throughout this paper, we use the $L^\infty$ norm on $\mathbb{C}^\ell \cong \mathbb{R}^{2\ell}$,
\[
\|z\| = \max_{j} \{ |\Re z_j|, |\Im z_j| \},
\]
so that ``dyadic boxes'' are balls in this metric which have centers with dyadic rational coordinates and a dyadic rational radius. We use $\cN_{\delta}(C)$ to refer to a $\delta$-neighborhood of a set $C$.

\begin{definition}
Let $D\subset \mathbb{C}^\ell$. We say a function $F:D\to \sC_m$ is computable if there exists a Turing machine $T=T(\psi,n)$ that, given $x\in D$ and any oracle $\psi$ of $x$, outputs a $2^{-n}$-approximation $T(\psi,n)=\Psi_n$ of $F(x)$. 
\end{definition}
As before, computable functions $F:D\to \sC_m$ are continuous. Finally, we introduce the notion of lower semi-computable open sets in $\mathbb{C}^\ell$.
\begin{definition}
An open set $U \subset \mathbb{C}^\ell$ is \emph{lower semi-computable} if there exists a Turing machine which outputs  a countable (possible infinite) sequence of open dyadic balls whose union is $U$.
\end{definition}

\subsection{H\'enon mappings and polynomial diffeomorphisms}

Generalized \Henon mappings are of the form
\[
f(z,w) = (p(z) - a w, z),
\]
where $p$ is a monic complex polynomial of degree $d \geq 2$ and $a$ is a non-zero complex number.
 Generalized \Henon mappings are polynomial automorphisms of $\mathbb{C}^2$ of dynamical degree $d$ (see \cite{BS1,FM}). We denote by $\mathcal{H}$ the space of generalized \Henon maps and by $\mathcal{H}_d$ the space of generalized \Henon maps where the polynomial $p$ has degree $d$.
We recall the definitions of the invariant sets $K^\pm$, $J^{\pm} = \partial K^{\pm}$, $J=J^+\cap J^-$ and $K=K^+\cap K^-$ from Section 1. We call $J^\pm$ the forward/backward Julia set of $f$, and $J$ is called the Julia set of $f$.
These sets provide the basic dynamical decomposition used in the algorithmic constructions of Section~\ref{sec:results}.

\section{Main Results on computability for generalized \Henon mappings}
\label{sec:results}

Our goal of this section is to establish:

\begin{theorem} \label{thm:main2}
Let $H$ be a generalized complex \Henon mapping such that any Fatou components are attracting basins. Then the Julia set $J_H$ of $H$ is computable.
\end{theorem}

Recall we let $\mathcal{H}$ denote the space of generalized \Henon mappings 
$
f : \mathbb{C}^2 \to \mathbb{C}^2.
$
Set
\[
\HNNP := \{ f \in \mathcal{H} :  \text{  any Fatou components of $f$ are attracting basins} \}.
\]
(We use ``C'' for ``computable''). 
 
We may assume without loss of generality that the norm of the Jacobian of $f$ is either smaller or equal than $1$, because otherwise we can simply consider $f^{-1}$. 

We now state a quantitative version of Theorem~\ref{thm:main2}, which yields computability of $J_f$.

\begin{proposition} \label{prop:main}
 There exists a Turing machine which, on input $N \in \mathbb{N}$, and oracle access to $f\in \HNNP$, outputs a finite collection of dyadic boxes $\Psi_N$ in $\Ct$ such that
\[
d_H(\Psi_N, J_f) \leq 2^{-N}.
\]
Moreover, for any $N$, the algorithm halts.
\end{proposition}
We note that Proposition \ref{prop:main} is formally a stronger result compared to Theorem \ref{thm:main2} as it includes that there is one Turing machine which computes the Julia set of all $f\in \HNNP$.

The proof of Proposition~\ref{prop:main} is constructive, and is based on the algorithm described in Subsection~\ref{sec:algorithm}. First, in the following subsection, we establish some helpful lemmas on computability aspects for \Henon maps, which we state separately as they may be of some independent interest.

\subsection{Computability and \Henon Lemmas}
\label{sec:computabilitylemmas}

The foundation of our approach is to exploit the following dynamical dichotomy.

\begin{lemma} \label{lem:dichotomy}
Let $f \in \HNNP$. Then one and only one of the following holds:

\begin{enumerate}
\item $f$ has an attracting periodic point, in which case ${\rm int}\, K^+ \neq \emptyset$, and every point of ${\rm int}\, K^+$ lies in an attracting basin.  

\item $f$ has no attracting periodic points, in which case $K^+=J^+$, and for any saddle periodic point $\alpha$,
\begin{equation}\label{BS2}
    J^+=\overline{W^s(\alpha)}.
\end{equation}
\end{enumerate}
\end{lemma}

\begin{proof}
(1) follows immediately from the definition of $\HNNP$. 

For (2), if $f\in \HNNP$ has no attracting periodic points,
there can be no Fatou components at all. Hence
${\rm int}\,K^+=\emptyset$, thus
$
K^+=J^+.
$

Now, let $\alpha$ be a saddle periodic point.
By Bedford--Smillie \cite{BS2},
$
J^+=\overline{W^s(\alpha)}.
$
(in fact, this is true for all polynomial diffeomorphisms of dynamical degree $d>1$).  

Finally, note the two alternatives are mutually exclusive.
\end{proof}
In case (2), periodic points need not all be saddle points.
In principle, for example, Cremer or semi-Cremer periodic points may occur.
However, this does not affect the proof since the algorithm uses
the approximation of $J^+$ by stable manifolds of saddle points.
It is conjectured that Cremer points are in general not associated with Fatou components but, to the best of our knowledge, this has not yet been proven.

\begin{lemma} \label{lem:Wsloc-saddle-computable}
Let $f\colon \mathbb{C}^2 \to \mathbb{C}^2$ be a generalized complex \Henon map. Let $\alpha \in \mathbb{C}^2$ be a saddle periodic point, and assume without loss of generality (by replacing $f$ with an iterate) that $\alpha$ is fixed. Then the local stable manifold $\Wslocalpha$ is computable. Let $V \subset \mathbb{C}^2$ be a closed box which contains $\Wslocalpha$.
Then there exists Turing machine which, on input $m,n\in\bN$ computes the set
\[
f^{-m}(\Wslocalpha) \cap V
\]
at precision $2^{-n}$.
\end{lemma}

\begin{proof}Variants of this Lemma are well-known, e.g., \cite{GZB}. Therefore we provide only a sketch.
Since $\alpha$ is a saddle fixed point, the derivative $Df(\alpha)$ has one eigenvalue of modulus less than one and one greater than one.

By the stable manifold theorem (see, e.g., \cite{KatokHasselblatt}), there exists a local stable manifold $\Wslocalpha$, which may be represented as the graph of a holomorphic function over the stable eigenspace.

This graph can be constructed via a graph transform procedure, which iteratively maps candidate graphs under $f$ and contracts in a suitable function space. Since $f$ and $Df$ are computable, each step of this iteration can be carried out with arbitrary precision, and the contraction guarantees effective convergence. Thus $\Wslocalpha$ is computable as a compact subset of $\mathbb{C}^2$.

To compute $f^{-m}(\Wslocalpha) \cap V$, to  precision $2^{-n}$, we apply the box-image algorithm that is described in Step (3) of Algorithm 3.2 of \cite{BoydWolf-Skew1} to the inverse map $f^{-1}$. At each step, we approximate the image of finitely many sub-boxes, controlling the error using Lipschitz bounds as in Lemma 2.16 and Corollary 2.17 of \cite{BoydWolf-Skew1}.

The intersection of these images with the fixed compact set $V$ is trivially computable. Therefore the sets
\[
f^{m}(\Wslocalpha) \cap V
\]
are computable.
\end{proof}

\begin{lemma} \label{lem:Wsloc-attr-computable}
Let $f\colon \mathbb{C}^2 \to \mathbb{C}^2$ be a generalized complex H\'enon map. If $f$ has an attracting periodic cycle, then its basin of attraction is lower semi-computable.
\end{lemma}

\begin{proof}
Let $\alpha$ be an attracting periodic point, and assume without loss of generality (by replacing $f$ with an iterate) that $\alpha$ is fixed. We then have 
$
|||Df^l(\alpha)||| < 1
$
for some computable $l\in \bN$.
By continuity of $Df^l$, there exists a box  $U$ centered at $\alpha$ such that $f^l$ is uniformly contracting on $U$. Such a box can be computed using derivative bounds as in Lemma 2.16 of \cite{BoydWolf-Skew1}. We also may assume without loss of generality $l=1$ since otherwise we replace $f$ with $f^l$.
It follows that the basin of attraction of $\alpha$ is given by
\[
\mathcal{U}(\alpha) = \bigcup_{n \geq 0} f^{-n}(U).
\]

For any $n$, we can compute an approximation of $f^{-n}(U)$ using the box-image algorithm applied to $f^{-1}$, with any given precision.

This yields a sequence of computable sets
\[
\mathcal{U}_n := \bigcup_{j=0}^n f^{-j}(U),
\]
with $\mathcal{U}_n \subset \mathcal{U}_{n+1}$ and
$
\bigcup_{n} \mathcal{U}_n
$
coincides with $ \mathcal{U}(\alpha)$.
We conclude that $\mathcal{U}(\alpha)$ is 
lower semi-computable.
\end{proof}

\subsection{Algorithm for computing the Julia set of a  generalized \Henon mapping in $\HNNP$}
\label{sec:algorithm}

We now present the algorithm which computes the Julia set $J = J_f$ for $f \in \HNNP$. Proof of its correctness, and halting, will follow in the subsequent subsection. 

\begin{algorithm}[\textbf{Computing a $2^{N}$-approximation of $J_f$}]
\label{alg:main}
Assume we have been given oracle access to $f \in \HNNP$, and a desired precision $2^{-N}$ for some $N \in \mathbb{N}$.

\textbf{Step 1.} (boxes)
Compute a large, dyadic box $V = [-R,R]^4 \subset \mathbb{C}^2$ such that $K_f$ is contained in $V$ (see \cite{BoydWolf-Henon} for a discussion of how to do this).

\medskip

Partition $V$ into a (uniform) grid of dyadic boxes $\mathcal{B}_N = \{B\}$ of sidelength $\varepsilon_N := 2^{-N}$.

\medskip

\textbf{Step 2.} (initialize $k$-loop)
Formulate increasing sequences of positive integers
\[
n_k, \; m_k, \; \ell_k, \; q_k,  \to \infty,
\]
and a decreasing sequence of dyadic rationals $ \xi_k \to 0$,
such that
$\xi_1 \ll \ep_N.$ 

Set $k=1$. 

\textbf{Step 3.} (attracting basins)
First, compute all periodic points of period at most $q_k$  (to a precision which is small compared to $\xi_k$, like ${2^{-2^k}}\xi_k $).

Then, for each periodic point, test the eigenvalues of the derivative (of appropriate iterate) norms and mark which, if any, periodic points are determined to definitively be attracting.

If no attracting periodic points are found, stop and proceed to the next Step. Else,
for each attracting periodic point, $\alpha$, that is found:

\begin{enumerate}
\item[(a)] Compute a neighborhood $U(\alpha)$  contained in the immediate basin of attraction of $\alpha$ (as in the proof of Lemma~\ref{lem:Wsloc-attr-computable}).

\item[(b)] As in the proof of Lemma~\ref{lem:Wsloc-attr-computable}, compute approximations of $\alpha$'s basin of attraction of the form
\[
\mathcal{U}_k(\alpha) := \bigcup_{i=0}^{\ell_k} f^{-i}(U(\alpha)),
\]
 where first we compute an approximation to each set $f^{-i}(U)$ with accuracy $\xi_k$ (using Lipschitz bounds to control errors) and then since we're attempting to exhaust $\cU_k(\alpha)$ from the inside, we refine the approximation by removing points within $\xi_k$ of the boundary of the approximation. That produces sets $\cU_k(\alpha)$ which definitely lie in ${\rm int}\, K^+$.

\end{enumerate}

Define $\mathcal{A}_k$ by taking the union of the sets $\mathcal{U}_k(\alpha)$ over all identified attracting periodic points $\alpha$.
it follows that $\cA_k \subset {\rm int}\, K^+$. 
Note $\mathcal{A}_k$ may be be empty; for example, if there are no attracting periodic points found in this step (for this $k$; it could be, for example, that all attracting periodic points are of period higher than $q_k$).

\medskip

\textbf{Step 4.} (escaping/attracting basins)

First, compute a Lipschitz bound $L_k$ for $f^{n_k},$ 
which works for all points in $V$ within a distance of $\ep_N$ of each other (find this bound as is done in Lemma 2.16 and Corollary 2.17 of \cite{BoydWolf-Skew1}).

For each $B \in \cB_N$, we want to subdivide $B$ into a uniform grid of dyadic sub-boxes $B'$, where the size of the sub-boxes is determined as follows. 

By choice of $L_k$, we have:
\[
\mathrm{diam}(f^{n_k}(B')) \leq L_k \cdot \mathrm{diam}(B').
\]

We require $B'$ so small that 
$L_k \cdot \mathrm{diam}(B') \leq \ep_N$.

Now, for each $B\in \cB_N$,  
then for each sub-box $B'$ of  $B$, let $c_{B'}$ denote its center. Then:

\begin{enumerate}

\item[(a)] from $i=1$ up to $i\leq n_k$, test whether 
the image of the center ($f^{i}(c_{B'})$) lies outside of $V$ 
 {\em and} if so, whether this image is a distance 
 $>(L_k+1) \varepsilon_N$ 
 from $\partial V$. 

If the answer to both questions  is yes for $B'$, for some $i$, then 
we mark $B'$ as an escaping-type sub-box (for level $k$).

\item[(b)] Else, if the answer to either question in (a) is no, then (if $\cA_k\neq \emptyset$) from $i=1$ for $i\leq n_k$, test whether $f^{i}(c_{B'}) \in \cA_k$; {\em and} if so, whether it is a distance 
$>(L_k+1) \varepsilon_N$
from $\partial \cA_k$. (This boundary distance is easy to calculate, as $\cA_k$ is a union of boxes.) 

If the answer to both of these questions is yes for some $i\leq n_k$, then 
we mark $B'$ as an attracting-type sub-box  (for level $k$). 

\end{enumerate}

Now:

(i) Once we have examined all $B'$s in a $B$, if {\em all} $B'$s in a $B$ were marked as the same type (escaping or attracting), then we classify $B$ as that type  (for level $k$).

(ii) On the other hand, for a $B$ if it both had one escaping-type sub-box and another attracting-type sub-box, then we mark $B$ as $J^+$-type  (for level $k$). 

\medskip

Let $\mathcal{B}_k^{\mathrm{surv}}$ denote the collection of boxes $B$ {\em not} marked as escaping-type  (for level $k$). 
(This set cannot be empty.)

Finally, if $\cA_k=\emptyset,$ we proceed to Step 5. (Note this means no boxes were found of attracting or $J^+$-type  (for level $k$), which could happen even if there are attracting periodic orbits but our $k$ is simply not large enough yet).

Else, if $\cA_k\neq \emptyset,$ skip the next step and proceed to Step~6.

\medskip

\textbf{Step 5.} (saddle-based detection)

If we get to this step, then $\cA_k=\emptyset$, i.e., no attracting periodic point is detected for this $k$. This means no boxes were marked as $J^+$-type  (for level $k$) above, so we handle that in this step.
Proceed as follows:

\begin{enumerate}
\item[(a)] Compute a saddle fixed point $\alpha$ (pass to an iterate of $f$ if needed; every generalized \Henon map has saddle periodic points), to an accuracy appropriate to $k$ (and increasing with $k$). 
\item[(b)] Compute a local stable manifold $\Wslocalpha$, using Lemma~\ref{lem:Wsloc-saddle-computable}. (Again, accuracy increases as $k$ increases.)

\item[(c)] Compute an approximation of
\[
f^{-m_k}(\Wslocalpha) \cap V,
\]
as a union of closed boxes, to a total accuracy $\xi_k$.  See Lemma~\ref{lem:Wsloc-saddle-computable}.
Then shrink the set by removing the $\xi_k$-neighborhood of the boundary, and let the smaller set be denoted by  $\mathcal{W}_k$ (so that we are guaranteed that $\cW_k\subset f^{-m_k}(\Wslocalpha) \cap V$). 
\end{enumerate}

Now, for each $B \in \cBsurvk$, in this case, we  mark $B$ as $J^+$-type  (for level $k$) if 
$B \cap \mathcal{W}_k\neq \emptyset$.

\medskip

We remark that, in contrast to the attracting case, where openness permits a center-point test with a margin, the saddle-based detection in this Step relies on set-intersection tests with computable approximations.

Finally, look at all boxes $B$ in $\cBsurvk$. If there are boxes not marked as a type  (for level $k$), stop this loop, increase $k$ to $k+1$ and start back up at Step 3. (Note, for efficiency, if a box was marked as a type for level $k$, it is still that type for higher level you need only examine the boxes for higher $k$ that weren't yet classified.)

\medskip

\textbf{Step 6.} (backward computation of $K^-$)

If this is the first visit to this Step for this $k$, repeat from Step 3 with $f$ replaced by $f^{-1}$ (and $K^+/J^+$ replaced by $K^-/J^-$, etc.), (potentially) providing a second $k$-level classification the boxes $B$ in $\cB_N$ (or exiting if some boxes cannot be backward classified). (This step helps to compute approximations of the complement of $K^-$, ${\rm int}\, K^-$,  and $J^-$.)

If this is the second visit to this step, simply proceed to the next step. 

\medskip

\textbf{Step 7.} (current approximation and termination)

If we get to this step, then we were able to classify all boxes in $\cB_N$ with both a forward and a backward type: escaping, attracting, or $J^\pm$ (for some level $k$). 

Now, we define $\Psi_N$ as the collection of boxes $B \in \cB_N$ such that  $B$ was classified above as {\em both} $J^+$-type (in Step 4 or 5 for $f$) {\em and} $J^-$-type (in Step 4 or 5 for $f^{-1}).$

Finally, output $\Psi_N$ as our $2^N$-approximation of $J$ and halt.
 
\qed
\end{algorithm}

\subsection{Correctness and Halting for the Algorithm}
\label{sec:mainproof}

We  now turn to establishing correctness of the algorithm, and halting, to complete the proof of Proposition~\ref{prop:main}. 

We first address correctness (before turning to halting near the close of the subsection), stated for ease of exposition as the following proposition.

\begin{proposition} \label{prop:correctness}
Let $f \in \HNNP$ and $N \in \mathbb{N}$. If Algorithm~\ref{alg:main} outputs a finite collection of boxes $\Psi_N$, then
\[
d_H(\Psi_N, J_f) \leq 2^{-N}.
\]
\end{proposition}

 We provide the main steps of the proof of this proposition via a series of lemmas. 

In the lemmas below, we assume the hypotheses of Proposition~\ref{prop:correctness}: that we have a map $f\in \HNNP$, and we examine the steps of Algorithm~\ref{alg:main} one at a time.

We also only provide the lemmas for $f$, thus for $K^+$, $J^+$. The same arguments applied to $f^{-1}$ justify the parallel conclusions for $K^-, J^-$.

We first show correctness of escaping or attracting-type classification.

\begin{lemma} \label{lem:basins}
If a box $B$ is classified as escaping-type or attracting-type (see Step 4), for some loop index $k$, then for every $z\in B$ we have $d(z, J^+) > \ep_N$. 

In particular, for escaping-type, we have $B \cap \cN_{\ep_N}( K^+) = \emptyset$, and for attracting-type, we have $B \subset {\rm int}\, K^+$ and $B \cap \cN_{\ep_N}(\partial K^+) = \emptyset$.
\end{lemma}

\begin{proof}
The proofs of the two cases are very analogous so we write them in parallel, focusing on the escaping case, with the attracting modifications in parentheses.

Suppose $B$ is classified as escaping(or attracting)-type in Step 4, for a value of $k$.

By construction, this occurs only if, after subdivision of $B$ into small dyadic sub-boxes $B' \subset B$, each such $B'$ satisfies both of the following:

\begin{enumerate}
\item  diam$(f^{n_k}(B')) \leq \varepsilon_N$, as ensured by the Lipschitz estimate in Step~4, and

\item the image of the center $c_{B'}$ of $B'$ under $f^{n_k}$ lies  outside of $V$ (or inside of $\mathcal{A}_k$), {\em and} ``safely'' for the $k$-level, which means is at distance  
$>(L_k+1)\varepsilon_N$ from the boundary.
\end{enumerate}

Thus, for each such sub-box $B'$, the entire image $f^{n_k}(B')$ lies at least a distance 
$> L_k \varepsilon_N$ 
outside of $V$ (or inside of $\cA_k$). 
That is, $f^{n_k}(B') \subset \CC\setminus V$ (or $f^{n_k}(B') \subset \cA_k$) and 
\[
d_H\bigl(f^{n_k}(B'), \partial V \bigr) 
> L_k \ep_N
\ \ (\text{ or } 
d_H\bigl(f^{n_k}(B'), \partial \mathcal{A}_k\bigr)
> L_k \ep_N ).
\]
The above means we have $d_H\bigl(f^{n_k}(B'),J^+\bigr)
> L_k \varepsilon_N$. 
(Indeed, in the attracting case this holds since $\cA_K \subset \cBsurvk \subset V,$ and 
$\cA_k \subset {\rm int}\, K^+$ by construction.)

Now, to the conclusion. Suppose that there is a point $x\in B'$ such that $d(x,J^+) \leq \ep_N$. Let $y\in J^+$ be a point with $d(x,y)\le \ep_N.$ By construction of $L_k$, we then have 
$d(f^{n_k}(x), f^{n_k}(y)) 
\le L_k \ep_N.$ 
But $f^{n_k}(x) \in f^{n_k}(B')$ and by invariance of $J^+$, $f^{n_k}(y)\in J^+$. Hence, $d(f^{n_k}(B'), J^+) 
\le L_k \ep_N.$ 
But this contradicts what we established in the above paragraph. Hence there could not have been a point $x\in B'$ within $\ep_N$ of $J^+$, thus for all $x\in B'$,
$d(x,J^+) > \ep_N$. 

Since $B$ is covered by the sets $B'$, the same estimate holds for all $x\in B$.

This yields the desired conclusion. 

\end{proof}

Next we establish the correctness of the $J^+$-classification.

\begin{lemma} \label{lem:saddles}
If a box $B$ is classified as $J^+$-type (in Step 4 or Step~5), then for every $x\in B$, $d(x, J^+) \leq \ep_N$.
\end{lemma}

    \begin{proof}
    
\noindent \textbf{Case 1:} Suppose $B$ is classified as $J^+$ type in Step 4. That means when $B$ was split into sub-boxes $B'$, it had one escaping-type sub-box and another attracting-type sub-box. Applying the previous lemma to each of these sub-boxes, we see that $B$ intersects both ${\rm int}\, K^+$ and $\CC\setminus K^+$. Thus, $B$ intersects $\partial K^+ = J^+$.
Hence, every point in $B$ is within the diameter of $B$ (which is $\ep_N$) distance from $J^+$. 

\medskip

\noindent \textbf{Case 2:}  Suppose $B$ is classified as $J^+$ type in Step 5.
Recall, for each $B \in \cBsurvk$ we marked $B$ as $J^+$-type if 
$B \cap \mathcal{W}_k\neq \emptyset$, where $\cW_k$ was an approximation guaranteed to be contained in  $f^{-m_k}(\Wslocalpha)$.
Thus, for each $x\in B$
we have $d(x, f^{-m_k}(\Wslocalpha))\leq \mathrm{diam}(B) = \ep_N.$ 
Since  $\cW_k \subset f^{-m_k}(\Wslocalpha)) \subset W^s(\alpha)$ and $\overline{W^s(\alpha)} = J^+$, we have $d(x, J^+)\leq \mathrm{diam}(B) = \ep_N$. 
\end{proof}

As we noted above the prior lemmas, Step 6 repeats from Step 3 for $f^{-1}$ so the analogous correctness statements follow as above.

We will now show correctness of the algorithm.

\begin{proof}[Proof of Proposition~\ref{prop:correctness}]
Suppose the algorithm outputs a $\Psi_N$. Then it was defined in Step 7 as the collection of boxes $B \in \cB_N$ such that $B$  was classified for this $k$ as {\em both} $J^+$-type (in Step 4 or 5 for $f$) {\em and} $J^-$-type (in Step 4 or 5 for $f^{-1}).$ 

Also, if the algorithm produced an output, then we know all boxes in $\cB_N$ were classified (else the $k$-loop wouldn't have stopped).

By Lemma~\ref{lem:saddles}, applied to $f$ and $f^{-1}$, if for some $k$ we determine a box $B$ is $J^+$ (or $J^-$) type, then for every $x\in B$ we have 
$d(x,J^+)\leq \ep_N$ (or $d(x,J^-)\leq \ep_N$).

So by (b), this means every point of $\Psi_N$ satisfies $d(x,J^+) \le \ep_N$ as well as $d(x,J^-) \le \ep_N$, so $d(x,J) \le \ep_N$. Hence, $\Psi_N \subset \cN_{\ep_N}(J)$.  
Now we recall in Step 1 we defined $\ep_N= 2^{-N}$. Hence, we have  $\Psi_N \subset \cN_{2^{-N}}(J)$.  

On the other hand, we claim
$J \subset \cN_{\ep_N}(\Psi_N)$. 

Recall 
we concluded above that if $\Psi_N$ is produced, then
all boxes $B$ in $\cB_N$ are classified both forward and backward.  Thus, $V \setminus \Psi_N$ is the union of the collection of boxes $B$ in $\cB_N$ that were classified as escaping or attracting-type for either $f$ or $f^{-1}$.

By Lemma~\ref{lem:basins}, applied to $f$ and $f^{-1}$, if  $B$ is identified as escaping or attracting forward, then for every $z\in B$ we have $d(z, J^+) > \ep_N$. If $B$ is escaping or attracting backward, then $d(z, J^-) > \ep_N$.
Hence if either one of these cases holds, then $d(z, J) > \ep_N$. So, for all $B \in V\setminus \Psi_N$, and for any $z$ in such a $B$, we have $d(z,J) > \ep_N$. 

Hence,  $d_H(V\setminus \Psi_N, J) > \ep_N$, thus $J \subset \cN_{\ep_N}(\Psi_N)$.
Since $\ep_N=2^{-N}$, we have $J \subset \cN_{2^{-N}}(\Psi_N)$.

Finally, $J \subset \cN_{2^{-N}}(\Psi_N)$ and $\Psi_N \subset \cN_{2^{-N}}(J)$ implies $d_H(J, \Psi_N) \leq 2^{-N}$.
\end{proof}

Finally, we establish that our algorithm halts. 

\begin{lemma} \label{lem:halting}
For any $f\in \HNNP$, and any $N \in \NN$, Algorithm~\ref{alg:main} halts in finite time.
\end{lemma}

\begin{proof}
Let $\cB_N = \{B\}$ be the grid of boxes on $V$ defined in Step 1 of  Algorithm~\ref{alg:main}. The main algorithm is a loop through increasing positive integers $k$, where the algorithm halts if it finds a $k$ such that for the $k$-level, every box $B$ in $\cB_N$ can be classified, both forward and backward, as either escaping, attracting, or saddle-type. 
This works because our finite collection of (closed, bounded) boxes $B_N$ depends only on $N$, and not on $k$. 
       
Consider a $B$ in $\cB_N$. It satisfies one and only one of the following cases:

(i) $B \subset {\rm int}\, K^+$;

(ii) $B \subset V \setminus K^+$;

(iii) $B\cap J^+ \neq  \emptyset$.

Suppose (i) holds for $B$. Then the sequence of sets $f^n(B)$ converges to an attracting periodic cycle. For some finite $k$, that periodic cycle is detected in Step 3, and so the set $\cA_k$ is formed and grows with $k$. For sufficiently large $k$, our increasingly accurate approximation of the basin of this cycle must contain our increasingly accurate approximation of $f^{n_k}(B)$, since $n_k$ grows with $k$. 

Or, if (ii) holds for $B$ the argument is even simpler; as $n$ grows, $f^n(B)$ will eventually land sufficiently far outside of $V$, so as $k$ grows, $n_k$ grows and our approximation accuracies for $f^n(B)$ increase, and hence for a finite $k$ we find $f^{n_k}(B)$ out of $V$ by the desired margin $\ep_N$. 

Finally, suppose (iii) holds for $B$. Then there are two cases. 

Case (a): $f$ has (at least one) attracting cycle $\alpha$, in which case $W^s(\alpha)$ has nonempty intersection with $B$. This follows from  $J^+=\partial W^s(\alpha)$, see \cite{BS2}. Then, as in case (i) above, we will identify $\alpha$ for sufficiently large $k$. Thus, since the subdivision of $B$ into a grid of boxes $B'$ used to calculate $f^{n_k}(B)$ does become more fine as $k$ increases, for similar reasons as (i) and (ii) above, we will detect two sub-boxes $B'$ of $B$ where one is escaping-type and one is attracting-type, hence $B$ will be identified to be of $J^+$-type in Step 4.

Case (b): $f$ has no attracting cycles. This means that for each $k$, we do not detect a $\cA_k$, and thus always perform Step 5. Now, we know 
$f^{-m_k}(\Wslocalpha)$ tends to $W^s(\alpha)$ as $k\to \infty$, and $\overline{W^s(\alpha)} = J^+.$ Since we are assuming $B \cap J^+ \neq \emptyset$, and both $B$ and $J^+$ are closed, there will be points of $W^s(\alpha)$ arbitrarily close to (at least one point of) $B$. Again as above, for sufficiently large $k$ we will detect an intersection of $B$ and $\cW_k$ our increasingly accurate approximation of $f^{-m_k}(\Wslocalpha)\cap V$. 

Finally, since there are a finite number of such $B$'s (because the set $\cB_N$ doesn't change with $k$), and each is ``forward'' classified for some finite $k$, all will eventually be forward classified for a finite $k$. 

The proof that each $B$ also has one of the three backward types is completely analogous. 

Hence, the algorithm halts in finite time, after providing both forward and backward classification of each $B$ in $\cB_N$, and thus defining $\Psi_N$.
\end{proof}

\begin{proof}[Proof of Theorem~\ref{thm:main2}]
Combining Proposition~\ref{prop:correctness} with the halting Lemma \ref{lem:halting} establishes Proposition~\ref{prop:main}, and hence Theorem~\ref{thm:main2}.
\end{proof}

Thus, we have established the computability of the Julia set for generalized \Henon mappings for which all Fatou components are basins of attraction.

With Proposition~\ref{prop:main} established, we immediately get Corollary~\ref{cor:main}, that the Julia sets vary continuously.

We finally observe that the preceding construction immediately yields local computability statements for the dynamical sets, as follows.

\begin{corollary} \label{cor:KplusKminus-computable}
Let $f$ be as in Theorem~\ref{thm:main2}. Then for any bounded box $Q \subset \mathbb{C}^2$, the intersections
\[
K^\pm \cap Q \quad \text{and} \quad J^\pm \cap Q
\]
are computable.
\end{corollary}

\begin{remark}
It follows immediately that the set $K = K^+ \cap K^-$ is computable in this setting.
\end{remark}

We now turn to extensions.

\section{Generalizations}
\label{sec:generalizations}

In this Section   we describe how to establish Theorem~\ref{thm:main}, by extending the computability result above from generalized \Henon mappings to arbitrary polynomial diffeomorphisms of $\mathbb{C}^2$ of dynamical degree $d>1$.

Rather than repeating all details of Section~\ref{sec:results}, we specify here only where the arguments require justification for generalization. 

\begin{proof}[Proof of Theorem~\ref{thm:main}]
Since every polynomial diffeomorphism $f$ of $\mathbb{C}^2$ with dynamical degree $d>1$ is affinely conjugate to a finite composition of generalized \Henon mappings, there exist generalized \Henon maps $H_1,\ldots,H_s$ and an affine automorphism $\Phi$ such that
\[
f = \Phi^{-1} \circ H_1 \circ \cdots \circ H_s \circ \Phi.
\]
We first consider the case
\[
F = H_1 \circ \cdots \circ H_s,
\]
and then incorporate the conjugacy.

\subsection*{Computability lemmas.}

The computability statements used in Lemmas~\ref{lem:Wsloc-saddle-computable}--\ref{lem:Wsloc-attr-computable} depend only on effective control of iterates, derivative bounds on compact sets, and hyperbolic structure of attracting and saddle periodic points. These properties are preserved under finite compositions of generalized \Henon maps, and hence the conclusions of these lemmas extend to $F$ with modified constants depending only on the number of compositions.

In addition, the dynamical dichotomy of Lemma~\ref{lem:dichotomy} continues to hold for $F$. Consequently, the classification of phase space into escaping, attracting, and saddle regimes used in Algorithm~\ref{alg:main} remains valid.

\subsection*{Computability of iterates.}

Since each $H_i$ is a polynomial automorphism, both $H_i$ and $H_i^{-1}$ are computable. It follows that $F$ and $F^{-1}$ are computable, as are all iterates $F^n$ and $F^{-n}$. Thus the forward and backward orbit computations used throughout Algorithm~\ref{alg:main} remain effective in this setting.

\subsection*{Lipschitz control and forward images.}

The control of diameters of box images in of Algorithm~\ref{alg:main} relies on computable Lipschitz bounds for iterates of the map. For the composition $F$, the derivative satisfies
\[
DF = D H_1 \cdot D H_2 \cdots D H_s,
\]
and hence
\[
|||DF(z)||| \leq \prod_{i=1}^s |||D H_i(z)|||.
\]
Thus Lipschitz constants for $F$ on compact sets can be computed effectively from those of the constituent \Henon maps (see Lemma~2.16 and Corollary~2.17 of \cite{BoydWolf-Skew1}). Since $s$ is fixed, this only modifies the constants appearing in the algorithm, and the diameter estimates used there continue to hold.

\subsection*{Escape and basin detection.}

The characterization of escaping points and attracting basins depends only on forward and backward iteration and is therefore unchanged for $F$. In particular, the escape test and the basin detection procedure apply directly, using the computability of $F$ established above.

\subsection*{Saddle dynamics and stable manifolds.}

All non-attracting periodic points lie in the Julia set and infinitely many of them are saddle points. For any saddle periodic point $\alpha$, the local stable manifold $W^s_{\mathrm{loc}}(\alpha)$ is computable (Lemma~\ref{lem:Wsloc-saddle-computable}). Moreover, the sets
\[
F^{-m}(W^s_{\mathrm{loc}}(\alpha))
\]
can be approximated using the same box-image procedure applied to $F^{-1}$.

By \cite{BS2}, we have $\overline{W^s(\alpha)} = J^+$, and hence, backward iterates of local stable manifolds approximate $J^+$. The argument used in our main algorithm, based on subdivision and Lipschitz control, carries over directly, since it depends only on computable bounds for $F$ and not on the specific form of a single \Henon map.

\subsection*{Completion of the construction.}

Combining the preceding observations, the algorithm of Section~\ref{sec:algorithm} applies to the map $F$ with only modified constants, depending on the fixed number $s$ of factors. Thus we obtain arbitrarily accurate approximations of $J^+$ and $J^-$, and hence of $J = J^+ \cap J^-$.

\subsection*{Conjugacy reduction.}

Finally, let
\begin{equation}\label{eqFM}
f = \Phi^{-1} \circ F \circ \Phi,
\end{equation}
where $\Phi$ and $\Phi^{-1}$ are affine automorphisms of $\mathbb{C}^2$. Since $f$ and $F$ are conjugate via an affine map, we conclude that the dynamical results needed in the algorithm of Section 3.2  hold for $f$. The only missing part is  that based on having oracle access to the the coefficients of $f$, one can compute the radius of a box $V=V_f$  satisfying $K_f\subset {\rm int}\, V_f$. From \eqref{eqFM} we conclude that $f$ can be written in the form
\begin{equation}\label{eqFM2}
f(z,w) = (A(a z + b w)^d + P(z,w), B(a z + b w)^d +Q(z,w)),
\end{equation}
where $d={\rm deg}\, F$, $A,B\in \bC$, $P,Q$ are complex polynomials of degree strictly smaller than $d$, and $a,b$ are the coefficients of the first coordinate of the linear part of $\Phi$. We note that at most one of the products $Aa$ and $Bb$ can be zero. Therefore, by evaluating the coefficients of $f$ at increasing precision we can  guarantee that either $Aa$ or $Bb$ are non-zero.
Without loss of generality we consider the case $Aa\not=0$. The case $Bb\not=0$ can be treated entirely analogous with $w$ taking the role of $z$ in the second coordinate of $f$. 
Next we approximate $Aa^d$ and $Ab^d$ at sufficiently high precision and denote these approximations by $a'$ and $b'$. At each approximation step we define $\Psi(z,w)=\Psi_{c',d'}(z,w)=(a'z+b'w,c'z+d'w)$, where $c',d'\in \bC$. 
By increasing the accuracy of the approximations $a'$ and $b'$ and at each step preforming a parameter search optimization algorithm for $c'$ and $d'$ we can assure that $g=\Psi\circ f\circ \Psi^{-1}$ is of the form
\begin{equation}\label{eqFM3}
g(z,w) = (Cz^d + R(z,w), S(z,w)),
\end{equation}
where $R$ and $S$ are polynomials with degree at most $d$ and all coefficients of degree $d$ terms of $R$ and $S$ are as small as needed  compared to $C$.
It follows that $g$ satisfies the criteria for the filtration lemma radius computation (see Lemma 2.1 in \cite{BS1}) and we can compute a radius of a box $V_g$ such that $K_g\subset {\rm int}\, V_g$. Since $g$ and $f$ are conjugate via $\Psi^{-1}$ it follows that $K_f=\Psi^{-1}(K_g)$. Thus, we may compute the radius of a box $V_f$ containing $K_f$ in its interior from the coefficients of $\Psi^{-1}$. We conclude that $J_f$ is computable which completes the proof of Theorem~\ref{thm:main}.
\end{proof}

Observe also that Corollary~\ref{cor:KplusKminus-computable} also translates immediately to this setting, establishing the computability of $K$ for polynomial diffeomorphisms of $\Ct$ with dynamical degree greater than one. 

\section{Examples of non-computability}
\label{sec:non-computability}

In our previous work (\cite{BoydWolf-Skew1,BoydWolf-1Dim,BoydWolf-Henon}), we established computability of Julia sets for several classes of multi-dimensional dynamical systems under hyperbolicity assumptions, including polynomial skew products and complex \Henon mappings, which we extended to certain non-hyperbolic cases above. It is a natural next step to explore non-computability examples in higher-dimensional settings. 
Thus, we now establish some examples of non-computability of maps of more than one complex dimension.

These examples contrast with the computability results of Section~\ref{sec:results}, and show that the presence of neutral (in particular, semi-Siegel) dynamics can be an obstruction to computability.

\subsection{Non-computability for polynomial skew products}

In the case of polynomial skew products, the structure of the Julia set is sufficiently well understood to allow a direct reduction to the one-dimensional situation. In particular, the presence of a non-computable base dynamics immediately forces non-computability of the full Julia set.
(Recall that Julia sets of polynomials in the complex plane are computable if and only if there are no Siegel disks \cite{BravermanYampolsky2006, BY2009}.)

\begin{proposition} \label{prop:SkewNonComputableCase}
Let $F(z,w)=(p(z),q(z,w))$ be a polynomial skew product (for which $p$ and $q$ are polynomials of the same degree $d\geq 2$).
If $J(p)$ is non-computable, then $J(F)$ is non-computable.
\end{proposition}

The result follows directly from Jonsson's characterization of the Julia set $J(F)$ for polynomial skew products: $J(F)$ is the closure of the repelling periodic points, and also the closure of the union of fiberwise Julia sets ($J_z$, defined in terms of a fiberwise Green's function which itself is derived from a global Green's function, see \cite{Jonsson1999}; $J(F)$ is also the support of the measure of maximal entropy). Since the set of base coordinates occurring in $J(F)$ is dense in $J_p$, any algorithm computing $J(F)$ would yield an algorithm computing $J_p$, a contradiction.  

\subsection{Non-computability for complex \Henon mappings}

For \Henon maps, we find non-computability of Julia sets also arises from Siegel phenomena; specifically, semi-Siegel dynamics. Unlike for skew products, where non-computability is inherited from the base dynamics, the phenomenon here is intrinsically two-dimensional.

Let $\theta$ be an irrational number in $[0,1)$. We recall that $\theta$ is a \textit{Bryuno} (or Brjuno) number (or simply is Bryuno) if
\[
\sum_{n=0}^\infty \frac{\log q_{n+1}}{q_n}<\infty,
\]
where $q_n$ are the denominators of the convergents in the continued fraction expansion of $\theta$. It is well-known that the set of Bryuno numbers has full Lebesgue measure.

Let $H:\mathbb{C}^2 \to \mathbb{C}^2$ be a dissipative \Henon map with a semi-Siegel fixed point, i.e., a fixed point with multipliers $(\lambda,\mu)$ satisfying $|\lambda|=1$ with $\lambda=e^{2\pi i \theta}$, where $\theta$ is a Bryuno number, and $|\mu|<1$. In this situation, the dynamics locally splits into a neutral direction and a contracting direction. Geometrically, this produces an invariant ``tube'' (or cylinder-like region) consisting of stable sets which are exponentially attracted toward a one-dimensional Siegel disk in the neutral direction. In particular, every point in this tube converges toward the Siegel disk under forward iteration, and the fixed point itself lies in the interior of the forward filled set $K^+$.

Semi-Siegel H\'enon dynamics has been studied in detail in \cite{BS2,FS,YampolskyYang2021}.

This configuration is unstable under perturbations of the rotation number. If $\theta$ is perturbed to a nearby rational value, the semi-Siegel structure is disappears, and the corresponding periodic point moves into the Julia set $J$. Thus, in passing from irrational to rational rotation numbers, points which were previously in the interior of $K^+$ become boundary points.

\begin{proposition}\label{Prop52}
Let $(H_\theta)$ be a family of dissipative \Henon maps depending continuously on a parameter $\theta$, and assume that $H_{\theta_0}$ has a semi-neutral fixed point $\alpha$ with neutral multiplier  $\lambda=e^{2\pi i {\theta_0}}$, where $\theta_0$ is  Bryuno. Then the map $H_{\theta_0}$ has a semi-Siegel fixed point, while for rational $\theta$ the corresponding fixed point lies in the Julia set $J(H_\theta)$. Then the Julia set $J(H_{\theta_0})$ is not computable.
\end{proposition}

\begin{proof}
Since $\theta_0$ is Bryuno, $\alpha$ is a semi-Siegel fixed point of   $H_{\theta_0}$. Hence $\alpha$ lies in the interior of $K^+(H_{\theta_0})$.
Let $\theta_n \to \theta_0$ be a sequence of rational parameters. For each $\theta_n$, the corresponding periodic point $\alpha_n$ lies in the Julia set $J(H_{\theta_n})$, and hence on the boundary of $K^+(H_{\theta_n})$. By continuity of the family, $\alpha_n \to \alpha$ in $\mathbb{C}^2$.
It follows that the sets $J(H_{\theta_n})$ contain points arbitrarily close to $\alpha$, while $\alpha$ lies in the interior of $K^+(H_{\theta_0})$. Therefore, the Julia sets $J(H_\theta)$ do not vary continuously in the Hausdorff topology at $\theta_0$.

On the other hand, if $J(H_\theta)$ were computable at $\theta_0$, then by general results in computable analysis (see, e.g., \cite{BY2009}) the map $\theta \mapsto J(H_\theta)$ would be continuous  at $\theta_0$. This contradiction shows that $J(H_{\theta_0})$ is not computable.
\end{proof}

Next we show the existence of Siegel disk \Henon maps with computable parameters, thus extending the corresponding one-dimensional result of Braverman and Yampolsky \cite{BravermanYampolsky2006}. They consider the quadratic family
\[
p_\theta(z)=z^2+ e^{2\pi i\theta} z, \quad \theta \in [0,1),
\]
which has a neutral fixed point at $0$. If  $\theta$ is Bryuno, then $p_\theta$ has a Siegel disk centered at $0$. In particular, $p_\theta$ has a Siegel disk at the origin for $\theta$-values in a set of full Lebesgue measure. It is natural to conjecture that the non-computability of Siegel Julia sets $J_\theta$  stems from the non-computability of $\theta$ -- surprisingly, this is not the case. 

Braverman and Yampolsky  construct an explicit algorithm that computes $\theta$-values with a Siegel disk at the origin. In particular, there exist computable $\theta$ values with a noncomputable Julia sets $J_\theta$. 

We generalize the result of Braverman and Yampolsky to \Henon maps. Consider the family $(H_{a,\theta})$ of quadratic \Henon maps defined by
\begin{equation}
    H_{a,\theta}(z,w)=(p_\theta(z)-aw,z),
\end{equation}
where $\theta\in[0,1)$ and $a\in\bC\setminus\{0\}$ with $|a|<1$.
Clearly, $H_{a,\theta}$ has a fixed point at the origin, and the eigenvalues of $DH_{a,\theta}(0,0)$ are given by
\begin{equation}
\lambda_{1/2}=\frac{e^{2\pi i\theta}\pm \sqrt{e^{4\pi i\theta}-4a}}{2}.
\end{equation}
It follows from results in \cite{BS2} (also see \cite{FS}) that if $\lambda_1$ is Bryuno, then $H_{a,\theta}$ has a Siegel disk $D$ centered at the origin and the stable set $W^s(D)$ of $D$ is an invariant connected component of ${\rm int}\, K^+$. Here $D$ is a $H_{a,\theta}$-invariant disk in $\bC^2$ whose dynamics is conjugate to an irrational rotation with rotation number $\theta$ on the (one-dimensional) unit disk. We have the following:
\begin{proposition}
    Let $\xi\in (0,1)$ be a computable Bryuno number and let $\theta\in [0,1)$ be computable. Then if $\theta$ is close enough to $\xi$ there exists a computable $a=a(\xi,\theta)$ such that $H_{a,\theta}$ has Siegel disk $D$ at the origin with rotation number $\xi$. Moreover,  if $\theta$ is close enough to $\xi$ then $|a|<1$ and $W^s(D)$ is a Siegel cylinder of $H_{a,\theta}$.
\end{proposition}
\begin{proof}
A simple calculation shows that by choosing $a=\frac14\left[ e^{4\pi i \theta}- (2e^{2\pi i\xi} -e^{2\pi i \theta})^2\right]$ the eigenvalue $\lambda_1$ of $D H_{\theta,a}(0,0)$ is $e^{2\pi i \xi}$. Moreover, $a$ is computable since $\theta$ and $\xi$ are computable. Finally, if  $\theta$ is close enough to $\xi$ then $|a|<1$ follows from the definition of $a$.
\end{proof}
It is important to note that the non-computability result  for Siegel Julia sets in Proposition \ref{Prop52} stems from the fact that the Julia set moves discontinuously in the Hausdorff metric at a Siegel parameter. One might ask if a Siegel Julia set remains to be non-computable if we allow the Turing machine to have access to the information that $D H_{\theta,a}$ has a Siegel cylinder at the origin. The corresponding one-dimensional question has been positively answered by Braverman and Yampolsky who showed that there are computable parameters $\theta$ so that $p_\theta$ has a non-computable Siegel disk centered at the origin. With some additional work, it should be possible to extend this result to the two-dimensional \Henon family $H_{a,\theta}$.
Since a detailed analysis of this statement would exceed the scope of this paper, we leave it open for future work.

\section{Further directions}
\label{sec:future}

The results of this paper suggest several natural directions for further study.

First, it would be of interest to extend these methods to polynomial automorphisms in $\mathbb{C}^n$ for $n>2$. In earlier work \cite{BoydWolf-Henon}, we extended hyperbolicity-based computability results to this higher-dimensional setting, and it would be natural to investigate whether similar extensions hold in the absence of neutral dynamics. 

The role of neutral dynamics in higher-dimensional systems remains only partially understood. In contrast to the one-dimensional setting, neutral dynamics in several complex variables can exhibit substantially more complicated behavior, and it would be of interest to characterize more precisely which types of neutral phenomena lead to computability or non-computability of Julia sets. 
In particular, one may ask to what extent parabolic behavior admits a useful combinatorial description: in one dimension, finite combinatorial data governs parabolic dynamics and yields computable Julia sets, whereas in higher dimensions both the definition and structure of parabolic phenomena remain much less understood. Identifying natural classes of maps with meaningful combinatorial invariants, and determining whether such data suffices to guarantee computability, is an interesting direction for future work (see, e.g., \cite{RaduTanase2018,RaduTanase2019}).

Another direction is to study quantitative aspects of these results, such as the computational complexity of approximating Julia sets beyond mere computability. In earlier work \cite{BoydWolf-1Dim} (see also \cite{DY2018}), polynomial-time complexity results were obtained in certain one-dimensional settings; however, to our knowledge, no such results are known in several complex variables, either in the presence or absence of hyperbolicity. It would be of interest to determine whether comparable complexity bounds can be obtained in this higher-dimensional setting.

It would also be of interest to investigate the  computability of Julia sets in parameter families of polynomial diffeomorphisms. In \cite{BoydWolf-Henon}, we showed that the hyperbolicity locus for \Henon maps is lower semi-computable; it would be natural to investigate how such effective properties of parameter space interact with the presence of neutral dynamics.

\bibliographystyle{plain}
\bibliography{BW-References4}

\end{document}